\documentclass[12pt]{article}

\usepackage[utf8x]{inputenc}
\usepackage[english]{babel}
\usepackage{amssymb, amsmath, amsthm, verbatim, cite, lscape, longtable, graphicx, wrapfig, tikz, subcaption}
\usepackage[english]{babel}
\usepackage[colorinlistoftodos]{todonotes}
\usepackage[font=small,labelfont=bf]{caption}

\newtheorem{lemma}{Lemma}

\newtheorem{proposition}{Proposition}
\newtheorem{definition}[lemma]{Definition}

\title{Combinatorial properties of continuous graphs:\\ A survey of challenges, solutions and open problems}

\author{
Alexander Grigoriev\thanks{\texttt{a.grigoriev@maastrichtuniversity.nl}, Department of Data Analytics and Digitalisation, Maastricht University, The Netherlands}
\and
Katherine Faulkner
\thanks{\texttt{katherinefaulkner1@outlook.com}, Department of Data Analytics and Digitalisation, Maastricht University, The Netherlands}
}

\date{}

\begin{document}
\maketitle

\begin{abstract}
Inspired by notorious combinatorial optimization problems on graphs, in this paper we consider a series of related problems defined using a metric space and topology determined by a graph. Particularly, we present the Independent Set, Vertex Cover, Chromatic Number and Treewidth problems on, so-called, continuous or metric graphs where every edge is represented by a unit-length continuous interval rather than by a pair of vertices. If any point of any unit-interval edge is considered as a possible member of a hitting set or a cover, the classical combinatorial problems become trickier and many open questions arise. Notably, in many real-life applications, such a continuous view of a graph is more natural than the classic combinatorial definition of a graph. The contribution of this paper is twofold: i) we survey the known results for optimization problems on continuous graphs, and ii) we create a list of open problems related to the continuous graphs.  
\end{abstract}

%%%%%%%%%%%%%%%%%%%%%%%%%%%%%%%%%%%%%%%%%%%%%%%%%
\section{Introduction}
%%%%%%%%%%%%%%%%%%%%%%%%%%%%%%%%%%%%%%%%%%%%%%%%%

Motivated by numerous applications in transportation, telecommunication and engineering, this paper suggests an alternative definition of a graph and addresses classic combinatorial optimization problems under this new definition of graphs. Notice that in physical networks the joints and/or splits are not unique points for infrastructure installation, e.g., sensors can be installed at any point of a physical or geographical network. Therefore, restricting the definition of a graph vertex to an element of a finite or a countable ground set of objects does not seem to be entirely adequate. Also, restricting the definition of a graph edge to a pair of vertices in only adjacency relation sounds very limiting. In physical networks, one would rather treat an edge as a continuous interval with endpoints in the joints and/or splits, where every point of the continuous interval is a vertex of the network. Such an intuitive re-definition of a graph immediately calls for re-introducing many classic parameters, properties and measures on graphs, e.g., minimum vertex cover, maximum independent set, maximum clique, minimum dominating set and many others. In this paper we redefine and discuss several classic combinatorial problems formulated for the new continuous graphs. As a side note, all problems we consider in the paper are computationally intractable, i.e., $NP$-complete, in their original graph-theoretic variation.

%%%%%%%%%%%%%%%%%%%%%%%%%%%%%%%%%%%%%%%%%%%%%%%%%
\section{Preliminaries}
%%%%%%%%%%%%%%%%%%%%%%%%%%%%%%%%%%%%%%%%%%%%%%%%%

\begin{definition}
Given a finite set $V$ (similar to vertices of a graph in the classic graph theoretic sense), a \emph{continuous edge} $e_{u,v}$ between $u\in V$ and $v\in V$ is a (continuum) set of points homeomorphic to a closed interval with \emph{endpoints} in $u$ and $v$. For presentation purposes of this paper, we assume that all continuous edges are the rectifiable intervals of unit length though further problem descriptions and open questions can be easily generalized to continuous edges of possibly varying lengths. 
\end{definition}

\begin{definition}
Given $u,v\in V$ and a rectification homeomorphism $f:e_{u,v}\rightarrow [0,1]$, the \emph{distance} $d(p,q)$ between two points $p$ and $q$ of a continuous edge $e_{u,v}$ is defined as $d(p,q)=|f(p)-f(q)|$. 
\end{definition}

\begin{definition}
Given a finite set of endpoints $V$ and a set of continuous edges $E$ on $V$, a \emph{continuous graph} $\Gamma=(V,E)$ is a topological space on the point set $E$ with the metric $d(x,y):E\times E\rightarrow \mathrm{R}_+$, where the distance between any two points in $E$ is determined by the shortest path concatenation (summing up the distances) of continuous edges in $E$. Naturally, if a path connecting two points does not exist, the distance between these two points equals infinity. 
\end{definition}

Note that, identical to the above defined continuous graphs, metric graphs have been introduced in the topology context, see Mugnolo~\cite{M2021}, and in the geometry context as geometric realizations of the classic combinatorial graphs, see Bandelt and Chepoi~\cite{BC2008}. The correspondence between continuous/metric graphs and classic combinatorial graphs is presented in the following definition.

\begin{definition}
Given a (classic combinatorial) graph $G=(\tilde{V},\tilde{E})$ with vertex set $\tilde{V}$ and edge set $\tilde{E}$, the \emph{corresponding} continuous graph denoted $\Gamma(G)=(V,E)$ is a continuous graph with $V=\tilde{V}$ and $E=\{e_{u,v}:\ (u,v)\in \tilde{E}\}$. Respectively, given a continuous graph $\Gamma(G)=(V,E)$, the \emph{corresponding} (combinatorial) graph $G(\Gamma)=(\tilde{V},\tilde{E})$ is a graph on the vertex set $\tilde{V}=V$ and edge set $\tilde{E}=\{(u,v):\ e_{u,v}\in E\}$. 
\end{definition}

The following definitions translate some basic notions from classic combinatorial graph theory into helpful and very intuitive notions for continuous graphs. 

\begin{definition}
In a continuous graph $\Gamma=(V,E)$, any two endpoints $u, v\in V$ are referred as \emph{adjacent} if $e_{u,v}\in E$, otherwise they are referred as \emph{non-adjacent}. 
\end{definition}

\begin{definition}
In a continuous graph, an inclusion-maximal set of points on finite distance from each other is called a \emph{connectivity component}. 
\end{definition}

\begin{definition}
Given two points $p$ and $q$ of a continuous graph $\Gamma=(V,E)$ on a finite distance from each other, a $(p,q)$-\emph{path} is an inclusion-minimal concatenation of edges in $E$ starting at $p$ (or $q$) and finishing at $q$ (or $p$). A \emph{subgraph} of a continuous graph $\Gamma$ is a union of points in a finite set of paths in $\Gamma$. A (sub)graph is \emph{connected} if all its points are on finite distance from each other. A \emph{cycle} is a subgraph consisting of two $(p,q)$-paths which are disjoint everywhere but in $p$ and $q$, where $p$ and $q$ are some points in $E$. A cycle-free connected subgraph is called a \emph{subtree}, or a \emph{tree} for simplicity.
\end{definition}

\begin{definition}
For a point $p$ in a continuous graph $\Gamma$ and for a positive real number $r>0$, a \emph{ball} $B(p,r)$ is the inclusion-maximal subgraph of $\Gamma$ such that all its points are at distance at most $r$ from $p$. We refer to $p$ and $r$ as the \emph{center} and the \emph{radius} of ball $B(p,r)$, respectively.
\end{definition}

In a similar way, many other re-definitions of graph-theoretic concepts can be formulated for continuous graphs. For chromatic number see Faulkner~\cite{F2018}, for $p$-center see Megiddo and Tamir~\cite{MT1983}, for more recent studies on independent and dominating sets see Hartmann and Janßen~\cite{HJ2024} and Hartmann and Marx~\cite{HM2025}, for Chinese Postman and Travelling Salaesman Problems see Frei et al~\cite{ISAAC2024}. An interesting, though relatively straightforward, exercise is to prove equivalence of definitions above and combinatorial definitions for corresponding continuous and classic combinatorial graphs. 

%%%%%%%%%%%%%%%%%%%%%%%%%%%%%%%%%%%%%%%%%%%%%%%%%
\section{Open Problems}
%%%%%%%%%%%%%%%%%%%%%%%%%%%%%%%%%%%%%%%%%%%%%%%%%

%%%%%%%%%%%%%%%%%%%%%%%%%%%%%%%%%%%%%%%%%%%%%%%%%
\subsection{Maximum independent set}
%%%%%%%%%%%%%%%%%%%%%%%%%%%%%%%%%%%%%%%%%%%%%%%%%

In 1972, Karp \cite{K1972} introduced a list of twenty-one $NP$-complete problems, one of which was the problem of finding a \emph{maximum independent set} in a graph: Given a graph, one must find a largest set of vertices such that no two vertices in the set are connected by an edge. Such a set of vertices is called a maximum independent set of the graph and, in general, the problem of finding it is $NP$-hard. The problem remains intractable in many families of dense and sparse graphs. In the literature, for a graph $G$, the size of a maximum independent set is usually referred as $\alpha(G)$.

For a continuous graph $\Gamma=(V,E)$, the corresponding continuous counterpart of the problem can be formulated as follows. 

\begin{definition}
Given a positive real number $r>0$, a set of points $I\subseteq E$ is called an $r$-independent (or $r$-stable) set of $\Gamma$ if for any two points $p, q \in I$, it holds that $d(p,q) \geq 2r$. The $r$-independence number of $\Gamma$, further referred as $\alpha_r(\Gamma)$ or $\alpha_r$ for simplicity, is the cardinality of the maximum $r$-independent set in $\Gamma$.
\end{definition}

For illustration see Figure~\ref{fig:envelop}, where on the left side we present a maximum independent set in the combinatorial ``envelop'' graph while on the right side one can see a $1$-independent set in the corresponding continuous counterpart.

\begin{figure}[h]
\includegraphics[width=0.9\linewidth]{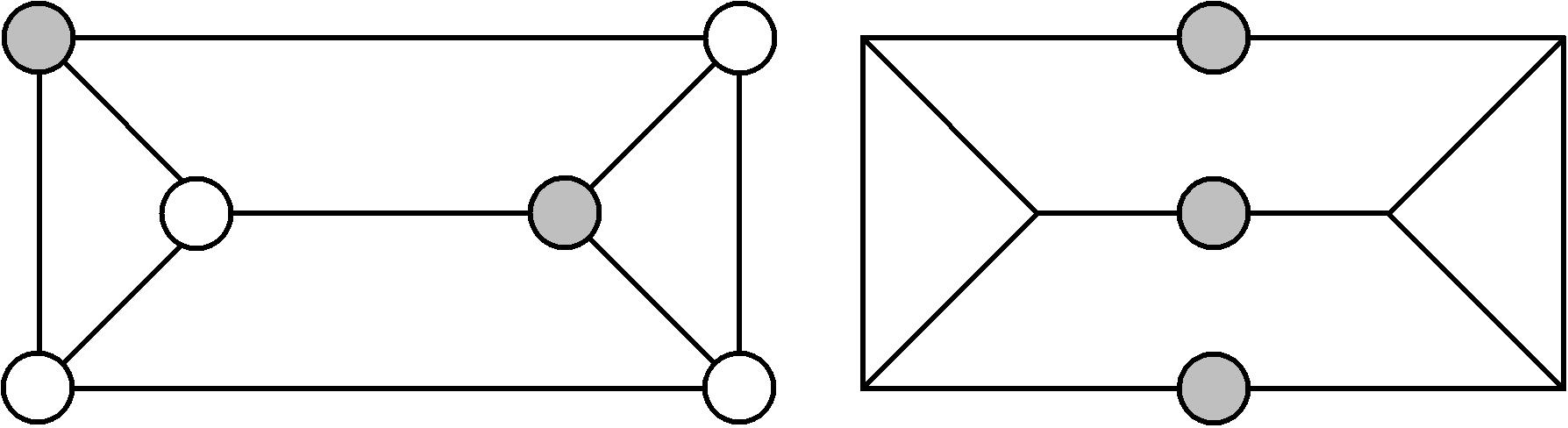} 
\centering
\captionof{figure}{Maximum independent set in combinatorial and in continuous graphs}\label{fig:envelop}
\end{figure}

Unsurprisingly, the maximum $r$-independent set problem on a continuous graph $\Gamma$ can also be seen as a ball packing problem where the task is to find a maximum cardinality set of balls of radii $r$ in $\Gamma$ such that the centers of the balls are mutually distant from each other by at least $2r$ or, in other words, the intersection of any two balls from the set is a finite number of points in $\Gamma$.

For a continuous graph $\Gamma=(V,E)$, a trivial upper bound on $\alpha_r(\Gamma)$ is $n+m/r$, where $n=|V|$ and $m$ is the number of (unit-length) edges in $E$. Here, the term $n$ dominates the number of possibly isolated endpoints in $\Gamma$, while $m/r$ trivially bounds the number of balls in any ball packing of $\Gamma$ with balls of radius $r$.

If $r=1$ and under an additional restriction that $I\subseteq V$, finding a maximum $1$-independent set in $\Gamma$ is equivalent to finding a maximum independent set in $G(\Gamma)$. When the additional restriction $I\subseteq E$ is relaxed, the problem becomes different. 

\begin{proposition}
For any continuous graph $\Gamma=(V, E)$, any $1$-independent set consists of at most $|V|$ points and this bound is tight.
\begin{proof}
Consider any $1$-independent set $S$ of $\Gamma$. Any point of $S$ is at distance at most $1$ from some endpoint in $V$, or, equivalently, every $1$-radius ball of the packing contains at least one endpoint in $V$. Let $S_1=S\setminus V$ and let $S_2=S\cap V$. Every $1$-radius ball in the packing having the center at $p\in S_1$ contains exactly one endpoint from $V$ and this endpoint does not belong to any of the other balls from the packing, neither from $S_1$ nor from $S_2$. Thus, $|S_1|\leq |V|-|S_2|$. The bound is tight, for example, when $\Gamma$ is a set of $n$ isolated (mutually non-adjacent) points. 
\end{proof}
\end{proposition}

Several nontrivial questions arise regarding $r$-independent sets and $\alpha_r$.

\paragraph{Discretization and tractability.} Is there a discretization of the edge set of a continuous graph preserving the $r$-independence number? Here, by discretization we mean any predetermined finite set of points in $\Gamma$, e.g., all endpoints and all points of the edges splitting the edges in small equal-length intervals. This question was addressed and solved by Hartmann and Lendl~\cite{HL2022} and Hartmann~\cite{H2022}. They show that there is a solution to the continuous problem where the independent set elements are situated in the points on rational distances from the endpoints of the graph. Moreover, they show that numerators and denominators of these rational distances are polynomially bounded in the input of the problem, which answers yet another non-trivial complexity question: Is there a polynomial size certificate (read polynomially encoded solution) for the maximum $r$-independent set problem in general continuous graphs? Here, the answer is yes, such a certificate does exist, based on the aforementioned rational distances. 

Now, when we know that the problem is in $NP$, the question is: For which values of $r$ the problem is solvable in polynomial time and for which values it becomes $NP$-hard? Consider, for instance, the case with large values for $r$, e.g., $r\geq \frac{n}{2}$ where $n$ is the number of endpoints of a given continuous graph. In this case, the problem is trivial as the diameter of a connected continuous graph, that is, the largest distance between any two points in a connected component of the graph, is at most $n-1$, making it impossible to find two points in a connected component that are distant from each other by $2r$.  For some small values of $r$, Grigoriev et al.~\cite{GHLW2021} present several complexity results. In particular, they prove that if $r=\frac{1}{b}$ or $r=\frac{2}{b}$ for some natural number $b$, the problem is solvable in polynomial time. On the negative side, they prove that if $r=\frac{a}{b}$ where $a$ and $b$ are integers such that $a\geq 3$ and $gcd(a,b)=1$, computing the maximum $r$-independent set is $NP$-hard. 

Clearly, the complexity of the problem is highly dependent not only on the nature of the number $r$ but also on the specific structure of the graph. Thus, the open question of this section is to have a picture of such a joint dependency.   

\paragraph{Relaxation gap for $r=1$.} First, we observe that the ratio between the $1$-independence number $\alpha_1(\Gamma)$ in the continuous graph and the size of the maximum independent set $\alpha(G)$ in the respective combinatorial counterpart $G=G(\Gamma)$ might be arbitrarily high. Consider a complete combinatorial graph $K_n$. Without loss of generality, let $n$ be even. Clearly, $\alpha(K_n)=1$ as all vertices are mutually adjacent and only one can belong to an independent set. Now, consider the continuous graph $\Gamma(K_n)$. Since $n$ is even, $K_n$ has a matching $M$ containing $n/2$ edges. In $\Gamma(K_n)$, consider a set of mid points of edges in $M$. By construction, the distance between any two such points is exactly 2. Therefore, the set of mid points of $M$ is a $1$-independent set in $\Gamma(K_n)$. Hence, $\alpha_1(\Gamma(K_n))\geq n/2$. It is unclear whether or not $n/2$ is a tight upper bound for the ratio $\alpha_1(\Gamma(G))/\alpha(G)$. We leave this as an open question of this section. Furthermore, we do expect that some structural properties of combinatorial graphs influence the ratio, and it is interesting to see how the ratio behaves in special cases/classes of graphs. 

\paragraph{Approximability.} Notice, unless $P=NP$, for any $\varepsilon>0$ there is no polynomial time algorithm producing an independent set of cardinality $n^{1-\varepsilon}$ for a general combinatorial graph on $n$ vertices. This is a straightforward consequence of the results by Zuckerman~\cite{Zuckerman2007} and H{\aa}stad~\cite{Hastad1999} that the maximum clique problem is hard to approximate within $n^{1-\varepsilon}$. This inapproximability result matches the trivial $1/n$-approximation, which is, e.g., constructing an independent set greedily starting with an arbitrary vertex. The question for the continuous version is whether this notorious inapproximability of the problem is still present in continuous graphs or the problem becomes somewhat easier to approximate? The high relaxation gap illustrated in the previous paragraph suggests that $\alpha(G)$ and the respective $\alpha_1(\Gamma)$ might differ significantly, implying that (in)approximability of the problems might also be different. Clearly, the (in)approximability, like tractability above, depends on the radius $r$. As mentioned above, the problem becomes polynomially solvable for large and some small values of $r$. For any rational $r\leq \frac{1}{2}$, every edge might contain at most $\lceil \frac{1}{2r}\rceil$ points from an $r$-independent set, and it will contain at least $\lfloor \frac{1}{2r} \rfloor$ points from the maximum $r$-independent set. The latter statement holds as for any edge of a continuous graph, regardless of other independent set points allocation, one can pick for the independent set the points in the interior of the edge on distance $r+i\times 2r$, $i=0,\ldots, \lfloor \frac{1}{2r} \rfloor -1$, from either of the edge endpoints. This leads to the performance guarantee of the solution of at least $1-\frac{1}{k+1}$, where $k=\lfloor \frac{1}{2r} \rfloor$. Since $r\leq \frac{1}{2}$, this number is at least $\frac{1}{2}$. The specific open question we ask in this paragraph reads: For any rational number $r>0$, is there a polynomial time approximation algorithm for the maximum $r$-independent set problem with performance guarantee bounded by a constant?    

%%%%%%%%%%%%%%%%%%%%%%%%%%%%%%%%%%%%%%%%%%%%%%%%%
\subsection{Vertex cover}
%%%%%%%%%%%%%%%%%%%%%%%%%%%%%%%%%%%%%%%%%%%%%%%%%

Another famous problem from the Karp's list~\cite{K1972} is the minimum vertex cover problem.  A \emph{vertex cover} of a graph is a subset of its vertices such that for every edge of the graph at least one of its endpoints is in the vertex cover. The vertex cover problem is to find a minimum size vertex cover. In the literature, the size of a minimum vertex cover of a graph $G$ is denoted by $\beta(G)$.

For a continuous graph $\Gamma=(V,E)$, the corresponding continuous version of the minimum vertex cover problem can be formulated as follows. 

\begin{definition}
Given a positive real number $r$, a set of balls $\mathcal{B}=\{B(p,r):\ p\in E\}$ is called an $r$-cover of $\Gamma$ if for any point $q\in E$ there is a ball $B\in \mathcal{B}$ containing $q$. The $r$-cover number of $\Gamma$ is the cardinality of the minimum $r$-cover in $\Gamma$ further denoted as $\beta_r(\Gamma)$ or $\beta_r$ for simplicity.
\end{definition}

If $r=1$ and under an additional restriction that all centers of the balls in $\mathcal{B}$ are in $V$ rather than in $E$, finding a minimum $1$-cover in $\Gamma$ is equivalent to finding a minimum vertex cover in $G(\Gamma)$. Again, if the additional restriction is relaxed, i.e., a center of a ball from $\mathcal{B}$ can be anywhere in $E$, the problem becomes different.

The following open questions would be interesting to addressed for the $r$-cover of a continuous graph.

\paragraph{The $\alpha-\beta$ duality.} There is a famous and very straightforward invariant in the classical graph theory: For every graph $G$ on $n$ vertices, it holds that $\alpha(G)+\beta(G)=n$. This is because a complement of an independent set in a combinatorial graph is a vertex cover. The same invariant from the classical graph theory still works for continuous graphs in case $r=1$, i.e., $\alpha_1(\Gamma)+\beta_1(\Gamma)=n$ for every continuous graph $\Gamma$. This was proven by Hartmann~\cite{H2022}, see Theorem 3.4.2 in that thesis. To illustrate this, consider the continuous ``envelope'' in Figure~\ref{fig:envelop}. Observe that $\alpha_1= 3$ and $\alpha=2$. Also observe that $\beta_1=3$ and $\beta=4$. For a $1$-cover of size $3$, take, e.g., the presented maximum $1$-independent set which is eventually also a cover. For simplicity of presentation, we do not provide a formal proof that $\alpha_1=\beta_1=3$ in the special case of the ``envelop'' example. We notice $\alpha_1+\beta_1=\alpha+\beta=6$ in this case. In general, the following question can be addressed: Given $r>0$, what is the bound on $\alpha_r(\Gamma)+\beta_r(\Gamma)$ for a continuous graph $\Gamma$ in terms of $n$ and $r$? Is the bound tight?

\paragraph{Discretization, tractability and $NP$-certificates.} The complexity of the minimum $r$-cover was extensively studied by Hartmann et al~\cite{HLW2022}. Particularly, they prove that for every rational $r>0$, the problem is contained in $NP$. Since this result is based on the subdivision of edges, we can affirmatively answer the question about existence of cover preserving discretization and existence of polynomial size certificates for the (continuous) minimum $r$-cover problem. Hartmann et al.~\cite{HLW2022} also show that the problem is polynomially solvable whenever $r$ is a unit fraction, and that the problem is $NP$-hard for all non-unit fractions $r$. As an interesting open problem, Hartmann et al.~\cite{HLW2022} suggested to study the complexity of the problem with an algebraic real $r$. Such problems are also contained in $NP$, and the conjecture of Hartmann et al.~\cite{HLW2022} was that the problem is $NP$-hard for all algebraic values $r$ that are not unit fractions. Recently, Hartmann and Lendl~\cite{HL2022} proved the conjecture for both cover and independent set problems even for all real $r>0$ that are not unit fractions. 

\paragraph{Relaxation gap for $r=1$.} Notice, the minimum $1$-cover problem is polynomially solvable by techniques from matching theory, see Hartmann et al.~\cite{HLW2022}, in contrast to the hardness of the classic minimum vertex cover problem. Similarly to the independent set problem above, the following natural open question arises: What is the highest ratio between $\beta(G)$ and $\beta_1(\Gamma(G))$ over all $G$? Following intuition from the matching theory and known results for the minimum vertex cover problem, we conjecture the ratio is at most 2. If this is true, the following example shows the tightness of the bound. Consider a complete combinatorial graph $K_n=(V,V\times V)$ on $n$ vertices and its continuous counterpart $\Gamma(K_n)$. To cover all edges of $K_n$, the minimum vertex cover $S$ should have at least $n-1$ vertices, otherwise the edge between two vertices from $V\setminus S$ is uncovered. Notice, any $n-1$ vertices of $K_n$ cover all edges in the graph. Therefore, $\beta(K_n)=n-1$. Consider $\Gamma(K_n)$ and let again, for simplicity of presentation, $n$ be even. Similarly to the construction for the independent set problem, consider any inclusion maximal matching $M$ in $K_n$ and the corresponding set of continuous edges in $\Gamma(K_n)$. Notice, the set of mid points in $M$ is a $1$-cover of $\Gamma(K_n)$ yielding $\frac{\beta}{\beta_1}\leq 2-\frac{2}{n}$. 

\paragraph{(In)approximability.} It is well known that the classic minimum vertex cover problem is $2$-approximable. To give a reader a very intuitive argument for that, consider the following folklore construction. Given a combinatorial graph $G$, let $M$ be an inclusion maximal matching in $G$, e.g., found greedily. Matching $M$ is a subgraph of $G$ and therefore a minimum vertex cover has at least $|M|$ vertices. Consider the set $S$ of all endpoints in $M$. Since $M$ is maximal, $S$ is a vertex cover of $G$. Notice, $|S|=2|M|$ that finishes the argument. There is also a negative result that, unless $P=NP$, there is no polynomial time approximation algorithm for the minimum vertex cover problem having performance guarantee better than $1.3606$; see Dinur and Safra~\cite{DinurSafra2005}. There is even some support to the claim that the minimum vertex cover problem is hard to approximate within $2-\varepsilon$ for any $\varepsilon>0$; see Khot and Regev~\cite{KhotRegev2008}. We suggest the following straightforward open question: For the hard cases of the problem (with the radius being a non-unit fraction, see~Hartmann et al~\cite{HLW2022}), what are the lower bounds for the inapproximability of the minimum $r$-cover problem? For insight and the first (in)approximability results see also Hartmann and Janßen~\cite{HJ2024}.

%%%%%%%%%%%%%%%%%%%%%%%%%%%%%%%%%%%%%%%%%%%%%%%%%
\subsection{Chromatic number}
%%%%%%%%%%%%%%%%%%%%%%%%%%%%%%%%%%%%%%%%%%%%%%%%%

The chromatic number problem is another problem on the Karp's list~\cite{K1972}.  The problem is defined as follows. Given a (combinatorial) graph $G=(\tilde{V},\tilde{E})$ and a positive integer $c>0$, the question is whether exist an assignment $f: \tilde{V}\rightarrow\{1,\ldots,c\}$ such that for every edge $(v,u)\in \tilde{E}$ it holds that $f(v)\neq f(u)$? If such an assignment does exist, the graph is called $c$-\emph{colourable} as every number from the set $\{1,\ldots,c\}$ is associated with a distinct colour and one is colouring the vertices of the graph in a way no two adjacent vertices receive the same colour. The problem of recognition of $c$-colourability is $NP$-complete even for $c=3$, see Dailey~\cite{D1980}. The \emph{chromatic number} of a graph $G$ is the minimum natural number $c$ such that $G$ is $c$-colourable. The chromatic number of a graph $G$ is denoted in the literature by $\chi(G)$. The problem of finding a chromatic number, as well as the independent set and vertex cover problems,  are among the most extensively studied problems in discrete mathematics and combinatorial optimization. 

Consider the continuous setting. For a given continuous graph $\Gamma=(V,E)$ and two positive integers $r>0$ and $c>0$ we define:

\begin{definition}
    Graph $\Gamma$ is $(r,c)$-colourable if there is a set of balls $\mathcal{B}=\{B(p,r):\ p\in E\}$ forming an $r$-cover of $\Gamma$ and an assignment $f: \mathcal{B}\rightarrow\{1,\ldots,c\}$ such that for every two balls $B,B'\in \mathcal{B}$ having non-empty intersection, i.e.  $B\cap B'\neq \emptyset$, it holds that  $f(B)\neq f(B')$.  The $r$-chromatic number of $\Gamma$, further denoted by $\chi_r(\Gamma)$ or $\chi_r$ for simplicity, is the minimum number of colours $c$ such that $\Gamma$ is $(r,c)$-colourable. 
\end{definition}

Similarly to the combinatorial problem, we associate the assignment $f$ with assignment of colours from the set $\{1,\ldots,c\}$ to the balls of the $r$-cover. Notice, should the centers of the balls in a $\frac{1}{2}$-cover of $\Gamma=(V,E)$ be the vertices in $V$, finding $(\frac{1}{2},c)$-coloring and $\frac{1}{2}$-chromatic number in $\Gamma$ are equivalent to finding $c$-colouring and chromatic number in $G(\Gamma)$, respectively. The colour assignment $f$ in the continuous version of the problem can be obtained from the colour assignment $f'$ of the combinatorial problem by assigning a ball centered at $v\in V$ the same colour as the vertex $v$ receives in $G(\Gamma)$, and vice versa. This holds because for any two adjacent vertices $v$ and $u$ of $G(\Gamma)$ receiving different colours in $f'$, the two $\frac{1}{2}$-radius balls with centers at $v$ and $u$ in $\Gamma$ do intersect in the middle point of the edge $(v,u)\in E$, and therefore receive different colours in $f$. 

Unlike the previous sections, when dealing with colouring of continuous graphs, we focus solely on the case $r=\frac{1}{2}$. The reason to consider only this radius is twofold. First, this case is a direct counterpart of the classic colouring and chromatic number problems on combinatorial graphs. Second, for $r=\frac{1}{2}$, the variety of challenging open problems is already very extensive. 

For this section, we came up with the following open questions. 

\paragraph{Colouring a complete continuous graph.} Let $K_n$ and $\Gamma(K_n)$ be a complete (combinatorial) graph on $n$ vertices and its continuous counterpart, respectively. A seemingly simple question to ask is: What is the $\frac{1}{2}$-chromatic number of $\Gamma(K_n)$? For example, in Figure~\ref{fig:K4 colouring proof}, a $(\frac{1}{2},2)$-colouring of $\Gamma(K_4)$ is presented where the continuous graph $\Gamma(K_4)$ is covered by the red and blue $\frac{1}{2}$-radius balls. 
\begin{figure}
    \centering
    \includegraphics[width=0.4\linewidth]{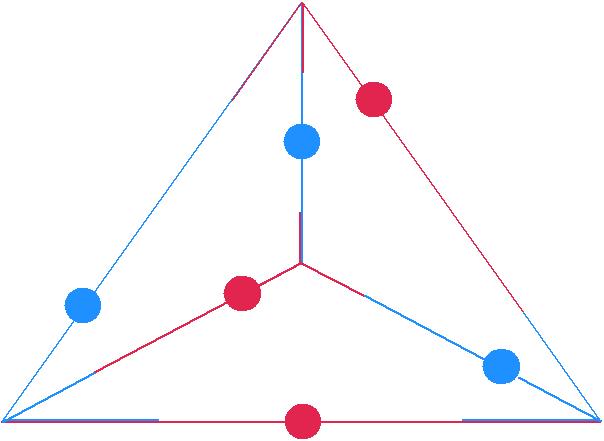}
    \caption{$(\frac{1}{2},2)$-colouring of $\Gamma(K_4)$}
    \label{fig:K4 colouring proof}
\end{figure}
We conjecture that the $\frac{1}{2}$-chromatic number of $\Gamma(K_n)$ is $\lceil\frac{n}{2}\rceil$. Though we cannot provide a rigorous proof for the lower bound of this conjecture, the following proposition establishes the upper bound. 

\begin{proposition}
$\chi_{\frac{1}{2}}(\Gamma(K_n))\leq \lceil\frac{n}{2}\rceil$ for any natural $n$.
\end{proposition}
\begin{proof}
Again, for simplicity of presentation, we prove the proposition for an even number $n$. Notice, this is without loss of generality as the bound in the proposition is rounded up and for odd numbers $n$ we can straightforwardly take the colouring of an even bigger graph $\Gamma(K_{n+1})$ involving $\frac{n+1}{2}$ colours. 

For an even number $n$, we construct $(\frac{1}{2},\frac{n}{2})$-colouring of $\Gamma(K_n)$ as follows. First, we take an arbitrary maximal matching in the graph, i.e., a set of $\frac{n}{2}$ mutually disjoint edges (intervals of unit length between the vertices in $V$). We assign to each of the edges of this matching a distinct colour. Thus, we have used already all $\frac{n}{2}$ colours available in the construction. Herewith we assume that edges of the matching are completely covered by distinctly coloured $\frac{1}{2}$-radius balls with the centers in the middle of the edges. 

Consider any two edges of the matching, say $e_{u,v}$ coloured blue and $e_{x,y}$ coloured red. Let us remind that  $u,v,x,y\in V$ and all these vertices are distinct. Consider four edges $e_{u,x}$, $e_{u,y}$, $e_{v,x}$ and $e_{v,y}$ which are not in the matching and therefore these edges are not yet covered by any balls and not yet coloured. On the edge $e_{u,x}$ on distance $\frac{3}{4}$ from $u$, and $\frac{1}{4}$ from $x$ respectively, we introduce a center of a new $\frac{1}{2}$-radius ball coloured blue. Likewise, on the edge $e_{v,y}$ on distance $\frac{3}{4}$ from $v$ we introduce a center of another blue $\frac{1}{2}$-radius ball. On edges $e_{u,y}$ and $e_{v,x}$, on distance $\frac{1}{4}$ from $u$ and $v$ respectively, we introduce the centers of two red $\frac{1}{2}$-radius balls. For illustration see Figure~\ref{fig:K4 colouring proof}, where the upper vertex is $u$, the central vertex is $v$, the left vertex is $x$ and the right vertex is $y$. We proceed with this construction for all pairs of edges in the matching.

Now, consider the entire set of all created balls. By construction, any two balls of the same colour are distant from each other by $\frac{1}{4}$. Moreover, all edges of $\Gamma(K_n)$ are covered by the constructed balls. Therefore, the resulting $\frac{n}{2}$-colouring of $\Gamma(K_n)$ is feasible that finishes the proof.     
\end{proof}

\paragraph{$(\frac{1}{2},3)$-colorability of continuous planar graphs.} It is well known that every loopless planar graph admits a $4$-colouring, see Appel and Haken~\cite{AH1977}, Appel, Haken and Koch~\cite{AHK1977}, or Robertson et al.~\cite{RSST1997}. A natural question arises: How many colours are needed for colouring of a continuous planar graph? From the discussion of the first open problem in this section we know that $\Gamma(K_4)$ is $(\frac{1}{2}, 2)$-colourable while $\chi(K_4)=4$. Notice, $K_4$ is a planar graph. This means, a general continuous planar graph might require less than 4 colours. We conjecture that every continuous planar graph is $(\frac{1}{2},3)$-colourable. To get better intuition for this question, it would be helpful to find a continuous planar graph requiring at least $3$ colours and to provide a formal proof of this lower bound of $3$.  A plausible candidate for such $(\frac{1}{2},3)$-colourability is the following graph. Consider $\Gamma(K_4)$ and any planar embedding of it. Let every face $f$ of the embedding receives an extra vertex $v_f$ in the interior of $f$ and edges are created between $v_f$ and the three vertices of $\Gamma(K_4)$ incident to that face $f$. One obtains a continuous planar graph on $8$ vertices and $16$ edges. Notice, extension of $(\frac{1}{2},2)$-colouring for $\Gamma(K_4)$ from Figure~\ref{fig:K4 colouring proof} of $\Gamma(K_4)$ is hardly possible for the constructed new graph as the colour assignment of edges incident to a face of the original $\Gamma(K_4)$ start conflicting with colour assignment of the new edges within that face. However, this does not mean that there is no other $\frac{1}{2}$-ball cover yielding $(\frac{1}{2},2)$-colourability of the constructed graph.

\paragraph{Approximability and bounds.} The discussion about inapproximability initiated in the previous two sections can be easily continued for the $\frac{1}{2}$-chromatic number problem. By Theorem 1.2. in Zuckerman~ \cite{Zuckerman2007}, for all $\varepsilon>0$, it is NP-hard to approximate the chromatic number within $n^{1-\varepsilon}$ for general combinatorial graphs. Again, this inapproximability result matches the trivial greedy $O(\frac{1}{n})$-approximation. The question for the continuous version of the chromatic number problem remains as above: Whether or not the same hardness of approximation remains for the $\frac{1}{2}$-chromatic number problem on continuous graphs? 

Design of approximation algorithms requires constructing upper and lower bounds on the $\frac{1}{2}$-chromatic number for continuous graphs. Regardless computational complexity, an interesting question to address is: Given a continuous graph $\Gamma$, can one specify a lower and an upper bounds for $\chi_{\frac{1}{2}}(\Gamma)$ in terms of a function of $\chi(G(\Gamma))$ and $n$? Any other bounds for $\chi_{\frac{1}{2}}$ as functions of other (especially, polynomial-time computable) parameters are also very much interesting.

%%%%%%%%%%%%%%%%%%%%%%%%%%%%%%%%%%%%%%%%%%%%%%%%%
\subsection{Treewidth}
%%%%%%%%%%%%%%%%%%%%%%%%%%%%%%%%%%%%%%%%%%%%%%%%%

In this section, we illustrate how more complicated graph theoretic notions can be translated to the continuous graphs. For example, consider a parameter \emph{treewidth} of a graph. The well-known usefulness of this parameter in combinatorial optimization is reflected in the following statement: On graphs of low (bounded by a constant) treewidth (almost) all classical hard problems, including Maximum Independent Set, Minimum Vertex Cover, Maximum Clique, become tractable, i.e., solvable in polynomial time, see e.g. the survey by Bodlaender~\cite{B1998}. Deciding whether a given graph has a bounded treewidth is also polynomially solvable, see Bodaender~\cite{B1996}. Therefore, for a large variety of combinatorial problems on quickly recognizable graphs of bounded treewidth, one can design efficient algorithms.  

While the definition of treewidth for a continuous graph is difficult to formulate directly, we will formalize it using a notion of $r$-\emph{bramble}: 

\begin{definition}\label{dfn:bramble}
Given a graph $G$ and a positive real number $r>0$, let $r$-\emph{bramble} $\mathcal{B}$ of a continuous graph $\Gamma(G)$ be a set of subtrees of $\Gamma(G)$, distant from each other by at most $r$, i.e., for any two subtrees $T_1$ and $T_2$ of the set, there is a point $p$ in $T_1$ and a point $q$ in $T_2$ on distance at most $r$. The \emph{order} of an $r$-bramble $\mathcal{B}$ of a continuous graph $\Gamma(G)$ is the minimum hitting set of $\mathcal{B}$, i.e., the minimum cardinality set $B$ of points in $\Gamma(G)$ such that every subtree from $\mathcal{B}$ contains at least one point from $B$. Finally, the $r$-\emph{bramble number} of $\Gamma(G)$ is the maximum order of an $r$-bramble over all $r$-brambles of $\Gamma(G)$.
\end{definition}

For the case $r=1$, Definition~\ref{dfn:bramble} is consistent with the definition of a bramble of a (combinatorial) graph introduced by Seymour and Thomas~\cite{ST1993}, where brambles first appeared under the name "screens". The theorem by Seymour and Thomas~\cite{ST1993} establishes the relation between the treewidth of a (combinatorial) graph and the bramble number: for any (combinatorial) graph $G$, the treewidth of $G$ is the bramble number of $G$ plus one. Now, one can define a concept which resembles the concept of treewidth in the setting of continuous graphs:

\begin{definition}
The treewidth of a continuous graph $\Gamma(G)$ is the $1$-bramble number of $\Gamma(G)$.
\end{definition}

Notably, the $1$-\emph{bramble number} is an interesting relaxation of the classic bramble number, and the $r$-\emph{bramble number} is an interesting generalization of it. On one side, this is exactly the way to estimate the treewidth of a continuous graph. On the other hand, recognizing whether the $1$-bramble number of a continuous graph is at most $k$, for a given small integer $k>0$, might be a very challenging and not necessarily tractable problem. We leave this as an open problem of this section.  

\subsection{Conclusion}
As shown by this paper, graph theory problems, ranging from classical problems to more intricate problems, can be reformulated in the continuous setting. This often represents reality more accurately and can lead to some interesting results. We invite readers to continue pursuing the potential of continuous graphs. 

\subsection*{Acknowledgements}
The authors thank Hans Bodlaender, Fedor Fomin, Jesper Nederlof and Gerhard Woeginger for very fruitful and constructive discussions and brainstorms during the Fixed-Parameter Computational Geometry 2018 workshop organized by Lorentz Center, Leiden, The Netherlands. Without their contributions this paper would not be possible. Special thanks go to Tim Hartmann for his very helpful comments, identifying many solved problems and pointing out multiple sources in the literature.

%%%%%%%%%%%%%%%%%%%%%%%%%%%%%%%%%%%%%%%%%%%%%%%%%

\end{document}